\documentclass[reqno,12pt]{amsart}
\usepackage{enumitem,stix}
\usepackage{epsf}
\usepackage[pagebackref]{hyperref}
\usepackage{mathtools}
\usepackage{ytableau}
\usepackage{todonotes}
\usepackage{verbatim}
\usepackage{amsmath}
\usepackage{amsopn}
\usepackage{latexsym}
\usepackage{amsfonts}
\usepackage{amsthm}
\usepackage{tikz}
\usepackage[foot]{amsaddr}
\usepackage[bottom]{footmisc}

\newtheorem{thm}{Theorem}[section]
\newtheorem{prop}[thm]{Proposition}
\newtheorem{lemma}[thm]{Lemma}
\newtheorem{fact}[thm]{Fact}

\theoremstyle{definition}

\newtheorem{rem}[thm]{Remark}

\newtheorem{example}[thm]{Example}

\newtheorem{definition}[thm]{Definition}

\DeclareMathOperator\SYT{SYT}

\DeclareMathOperator\Des{Des}

\newcommand{\sn}{{\mathcal S}_n}
\newcommand{\C}{{\mathcal C}}
\newcommand{\cmu}{{\mathcal C}_{\mu}}
\title{Major index on involutions}
\author{Eli Bagno}
\address{Department of Applied Mathematics, Jerusalem College of Technology, 21 Havaad Haleumi St., Jerusalem, Israel}
\email{bagnoe@g.jct.ac.il}
\author{Yisca Kares}
\address{Department of Mathematics, Bar-Ilan University, Ramat-Gan 52900, Israel}
\email{yisca95@gmail.com}

\begin{document}
\footnotetext{The second author partially supported by the Israel Science Foundation,grantno.1970/18}

\maketitle

\begin{abstract}
We find the range of the major index on the various conjugacy classes of involutions in the symmetric group $S_n$. 
In addition to indicating the minimum and the maximum values, we show that except for the case of involutions without fixed points, all the values in the range are attained. For the conjugacy classes of involutions without fixed points, we show that the only missing values are one more than the minimum and one less than the maximum.
\end{abstract}

\section{Introduction}

Let $\sn$ be the symmetric group on $n$ elements, let $\lambda=(\lambda_1,\dots,\lambda_k)$ be a partition of $n$ which will be identified with its Young diagram of the shape $\lambda$, and let $SYT(\lambda)$ denote the 
set of all standard Young tableaux of shape $\lambda$. We say that $i \in \{1,\dots,n-1\}$ is a {\it descent} of a permutation $\pi \in \sn$ if $\pi(i)>\pi(i+1)$. 

The major index of a permutation $\pi \in \sn$ , $maj(\pi)$, is the sum of the descents of $T$. 
Likewise, if $T\in SYT(\lambda)$ is a standard tableaux of shape $\lambda$ then $i \in \{1,\dots,n-1\}$ is a {\it descent} of $T$ if $i+1$ is located in $T$ in a row lower than $i$.  Again, the major index, $maj(T)$, is the sum of the descents. See a more elaborated exposition of these concepts in Section \ref{Background}. \\ 
Owing to the RSK algorithm which associates a pair of standard tableaux of the same 
shape to every permutation $\pi\in \sn$, these two versions of the major index are tightly connected. 

The generating function of the major index over the entire group ${\mathcal{S}}_n$ is known since the days of McMahon 
\cite{Mc} who showed that the distribution of the major index on all permutations in $\sn$  is the same as the distribution of inversions. Explicitly:
 $$\sum\limits_{\pi \in  \sn}q^{maj(\pi)}=[n]_q!$$ where $[n]_q=\frac{q^n-1}{q-1}$ and $[n]_q!=[n]_q\cdot [n-1]_q \cdots [1]_q$. 

On the other hand, the generating function $$SYT(\lambda)^{maj}(q):=\sum\limits_{T \in SYT(\lambda)}q^{maj(T)}=
\sum_{k \geq 0}m_{\lambda,k}q^k$$ has two elegant closed forms: one due to Steinberg \cite{Ste}, based on dimensions of irreducible representations of $GL_n(\mathbb{F}_q)$, and the other is based on Stanely's generalization of the hook formula \cite{Sta79}. It reads:
$$SYT(\lambda)^{maj}(q)=\frac{q^{b(\lambda)}[n]_q!}{\Pi_{c\in \lambda}[h_c]_q}$$

where $h_c$ is the hook length of the cell $c$ and $b(\lambda)=\sum\limits_{i > 0}(i-1)\lambda_i$.
The sequence $\{m_{\lambda,k}\}$  is called the {\it fake degree sequence} and has appeared in various algebraic 
and representation-theoretic contexts, such as the degree polynomials of unipotent $GL_n(\mathbb{F}_q)$
-representations due to Green \cite{Green}. 
Another context in which the fake degrees have a significant rule is the coinvariant algebra of $\sn$. 
More broadly, the $k$-th part of the coinvariant algebra of  ${\mathcal{S}}_n$ decomposes into 
irreducible representations, each of them appearing with multiplicity $m_{k,\lambda}$ . 

The range of the major index inside a single shape has been recently explored by Billey, Konvalinka and Swanson in \cite{Sw} (see Theorem 1.1. there).  
We concentrate here on the range of values the major index attains at the conjugacy classes of involutions (i.e. permutations of order less or equal to $2$). This will be done using the RSK correspondence between permutations and pairs of standard Young tableaux, which turns out to be a bijection with standard Young tableaux when restricted to involutions.

Recall that conjugacy classes of involutions are determined by the number of fixed points. According to a result by Sch\"{u}tzenberger \cite{S}, all the involutions of a given conjugacy class with $r$ fixed points are associated by RSK with standard tableaux of shapes having $r$ odd columns. 

The combination of these two results, together with an algorithm that runs through the standard Young tableaux of a fixed number of odd columns enable us to calculate the range of the major index over the entire conjugacy classes.   

The paper is structured as follows: First, we provide some background in the following section, including the definitions of the statistics we work with, the Young diagrams, the Young tableaux, and the Robinson-Schensted correspondence. Additionally, we present some preparatory lemmas in Section \ref{Background} to enhance the readability of our proof of the main theorem. Section \ref{main results} introduces our main theorem and the algorithm we use to demonstrate it, along with the proof's overarching idea. Further, we justify the algorithm's steps in Section \ref{detailed proofs}, where we provide detailed proofs for each step. Finally, we cover some special cases at the end of Section \ref{detailed proofs}.

%The paper is organized as follows: In the next section, we bring some background, including 
%the definitions of the statistics we work with, the Young diagrams, the Young tableaux, and the Robinson-Schensted correspondence. 

%We also provide in Section \ref{Background} some preparatory lemmas %which will make our proof 
%of the main theorem more readable. 

%Section \ref{main results} starts with our main theorem and presents the algorithm we use to 
%prove it and the idea of the proof, the details of which will be filled
%in Section \ref{detailed proofs} in the form of justification to the steps taken in the algorithm.  Finally, we treat some special cases at the end of Section \ref{detailed proofs}. 

\section{Background}\label{Background}

\subsection{Statistics on permutations}

We define here two of the well-known statistics on $\sn$, namely, the descent number and the major index. 

\begin{definition}
Let $\pi \in \sn$. The {\em descent set} of $\pi$ is:
$$Des(\pi)=\{i \in \{1,\dots,n-1\} \mid \pi(i)>\pi(i+1)\}.$$  
We also define $des(\pi)=|Des(\pi)|$. 

The {\em major index} of $\pi$ is defined by:
$$maj(\pi)=\sum\limits_{i\in Des(\pi)}{i}.$$
\end{definition}

\begin{example}
Let $\pi=5321476 \in \mathcal{S}_7$. Then $Des(\pi)=\{1,2,3,6\}$, $des(\pi)=4$ and $maj(\pi)=1+2+3+6=12$.  
\end{example}

\subsection{Young diagrams and standard tableaux}

A {\it Young diagram} is a finite collection of cells in the plane, arranged in left-justified rows, such that the row lengths are increasing. The sequence listing the numbers of cells in each row gives a partition $\lambda$ of a non-negative integer $n$.  The Young diagram is said to be of shape $\lambda$.
Also, let $\lambda'$ be the transpose of $\lambda$.
We often use Greek letters to denote the diagrams as well as their shapes (See figure \ref{diagram} for an example).

\begin{figure}[!ht]
\begin{center}

$\lambda=$ \begin{ytableau}
 {} & {} & {} & {}  \\
  {} & {} & {}\\
  {} \\
  {} \\
\end{ytableau}
$\lambda'=$ \begin{ytableau}
 {} & {} & {} & {}  \\
  {} & {}\\
  {} & {} \\
  {} \\
\end{ytableau}

\caption{\label{diagram} $\lambda=(4,3,1,1),\lambda'=(4,2,2,1)$}
\end{center}
\end{figure}

A {\it Young tableau} is a filling of a young diagram a by the numbers $1,\dots,n$. A tableau is {\it standard} if the numbers are increasing through rows and column. (see example in figure \ref{tableau}).  

\begin{figure}[!ht]
\begin{center}

\begin{ytableau}
 {\bf{1}} & {4} & {5} & {\bf{6}}  \\
  {\bf{2}} & {\bf{7}} & {9}\\
  {3} \\
  {8} \\
\end{ytableau}

\caption{\label{tableau} A standard tableau $T$ of shape $\lambda=(4,3,1,1)$, $Des(T)=\{1,2,6,7\},maj(T)=1+2+6+7=16$}
\end{center}
\end{figure}

For a shape $\lambda$, let $\SYT(\lambda)$ be the set
of standard Young tableaux of shape $\lambda$. 
The {\em descent set} of $T$ %a standard Young tableau $T$ %of size $n$ 
is
\[
\Des(T):=\{i \,\mid  i+1 \text{ appears in a lower row of } T \text{ than } i\}.
\]
Define also the {\it major index} of a standard Young tableau $T$ by $$maj(T)=\sum_{i \in Des(T)}{i}.$$

(See example in figure \ref{tableau}). 
\subsection{The RSK algorithm}
We shall make use of the Robinson-Schensted-Knuth (RSK) correspondence which maps each permutation $\pi \in \sn$ to a pair $(P_{\pi}, Q_{\pi})$ of standard Young tableaux of the same shape $\lambda$. A detailed description can be found, for example, in  \cite[Ch.\ 3.1]{Sagan} or in \cite[Ch.\ 7.11]{St2}.

The following is a crucial property of the RSK correspondence which will be vastly used in this paper.  

\begin{fact}\label{Fact}
For each $\pi\in \sn$ one has
$$Q_{\pi}=P_{\pi^{-1}}.$$ 
\end{fact}

Note also that the RSK correspondence is a $\Des$-preserving and hence also a $maj$- preserving bijection in the following sense:

\begin{fact}
For every permutation $\pi\in \sn$,
\[
\Des(P_\pi)=\Des(\pi^{-1}) \quad \text{ and } \quad
\Des(Q_\pi)=\Des(\pi).
\]
Consequently, we have also $maj(P_{\pi})=maj(\pi^{-1})$ as well as $maj(Q_{\pi})=maj(\pi)$.  
\end{fact}
 
It follows from Fact \ref{Fact} that $\pi$ is an involution if and only if  $P_{\pi}=Q_{\pi}$ so that by restricting the RSK correspondence to the set of involutions which will be denoted here by $I_n$, we get a $Des-$ preserving bijection from $I_n$ to the set of standard Young tableaux of order $n$, $SYT(n)$.
\subsection{Conjugacy classes of involutions}
Conjugacy classes in $\sn$ are determined by their cycle structures, which are partitions of $n$. 
The conjugacy classes of involutions in $I_n$ are of cycle structure $(2^k,1^r)$ (where $a^b$ means $b$ appearances of $a$) such that $0 \leq r \leq n$, $0\leq k \leq \frac{n}{2}$ and $2k+r=n$. In other words, conjugacy classes of involutions are determined by the number of fixed points. 

The following known result by Sch\"{u}tzenberger \cite{S}  gives a full description of the image of each conjugacy class of involutions under the RSK correspondence.

\begin{prop}\label{shuz}
An involution $\pi \in I_n$ has $r$ fixed points if and only if $P(\pi)$ has $r$ columns of odd length. 
\end{prop}

In light of this characterization, we present the following notations: 
\begin{definition}\label{r odd}
 The set of Young diagrams of size $n$ having exactly $r$ odd columns will be denoted by $D_n(r)$. 
 The set of standard Young tableaux of shapes taken from $D_n(r)$ will be denoted by $SYT_n(r)$.

\end{definition}

The discussion above can now be concisely formulated as follows: 

\begin{prop}\label{transfer to tableaux}
Let $\cmu$ be the conjugacy class corresponding to the partition $\mu=(2^k,1^r)$. Then the restriction of the RSK correspondence $$R:\cmu \rightarrow SYT_n(r)$$ is a bijection that preserves the major index, i.e. for each $\pi \in \cmu$ we have $maj(\pi)=maj(R(\pi))$. 
\end{prop}

The following definition will be needed in the sequel (note the slight 
change from a similar definition made in the introduction; we prefer this notation for the sake of convenience). 

\begin{definition}
For a shape $\lambda=(\lambda_0,\lambda_1,\dots,\lambda_u)$, let $$b(\lambda)=\sum\limits_{i=0}^u{i \lambda_i}.$$ 
\end{definition}

\begin{rem}\label{graphic blambda}
The numbers $b(\lambda)$, ($b(\lambda')$) can be easily calculated by writing for each $i$ the number $i$ inside each square of row (column) $i$ of $\lambda$ and adding up the numbers to get the rows (columns) sums respectively.  
\end{rem}

\begin{figure}[!ht]\label{rows and columns sums}
  \begin{ytableau}
 {0} &{0} &{0} &{0}\\
 {1} & {1} & {1} \\
 {2} & {2}     \end{ytableau}
,
 \begin{ytableau}
 {0} &{1} &{2} &{3}\\
 {0} & {1} & {2} \\
 {0} & {1}     \end{ytableau}

\caption{Calculating $b(\lambda)$ (left) and $b(\lambda')$ (right) by adding up rows and columns respectively}.
\end{figure}

The continuity of $maj$ inside the Young tableaux of a single Young diagram of a specific shape has been recently proven by Billey, Konvalinka, and Swanson in \cite{Sw} (see Theorem 1.1. there). 
The following is a reformulation of their result:

\begin{prop}\label{Max and Min}
Let $\lambda$ be a Young diagram. Then we have:

$$m(\lambda):={\rm Min}\{maj(T) \mid T \in SYT(\lambda)\}=b(\lambda).$$

$$M(\lambda):={\rm Max}\{maj(T)\mid T\in SYT(\lambda)\}={n \choose 2}-b(\lambda').
  $$

Moreover, every value between $m(\lambda)$ and $M(\lambda)$ appears at least once, except when $\lambda$ is a rectangle with at least two rows and columns, in which case the values $m(\lambda)+1$ and $M(\lambda)-1$ are missing. 
  
\end{prop}
 
The following lemma determines the diagrams of $D_n(r)$ which attain the minimum and the maximum of the major index. 
\begin{lemma}\label{min and max on $D_n(r)$}

Let $n=2k+r$.
\begin{enumerate}
    
    \item The minimum value of the major index on $D_n(r)$ is $k$. It is attained by the diagram $\lambda=(n-k,k)$. 
    
    \item The maximum value of the major index on $D_n(r)$ is ${n \choose 2}-{r \choose 2}$. It is attained by the diagram $\lambda=(r,1^{2k})$.

\end{enumerate}

\end{lemma}

\begin{proof}
\begin{enumerate}
    \item 
    Let $m={\rm Min}\{maj(T)\mid T\in SYT_n(r)\}$ and let $\lambda=(n-k,k)$. Since $n=2k+r$, it is easy to see that $\lambda$ has $r$ odd columns,
    so that $\lambda\in D_n(r)$. 
    By Proposition \ref{Max and Min}, $b(\lambda)=0\cdot (n-k) + 1\cdot k=k$ is the minimum value     
    of the major index on $\lambda$, hence $m\leq k$.
    Now, for each $\nu \in D_n(r)$, if $b(\nu)<k$ then the complement of the first line of $\nu$ contains less than $k$ squares and so the first line of $\nu$ must contain more than $n-k=r+2k-k=r+k$ squares. As the number of odd columns is $r$, there are more than $k$ squares in the first line which start columns of size at least $2$ so that $b(\nu)> k$, a contradiction.

    \item Let $M={\rm Max}\{maj(T) \mid T\in SYT_n(r)\}$ and let $\lambda=(r,1^{2k})$. It is easy to see that $\lambda \in D_n(r)$. We have $b(\lambda')=1+2+\cdots +(r-1)={r \choose 2}$ so we conclude from Proposition \ref{Max and Min} that $M\geq {n \choose 2}-{r \choose 2}$.
    Every $\nu \in D_n(r)$  must have at least $r$ columns, hence we must have $b(\nu')\geq {r \choose 2}$, and we are done. 
    
    \end{enumerate}

\end{proof}
We will show in the sequel that for most of the conjugacy classes of involutions, the range of the major index on these classes is an interval, but we start with the exceptional case: the class of involutions without fixed points. 
\begin{lemma}\label{missing}
Let $\C$ be the conjugacy class corresponding to the partition $\mu=(2^k)$.   
Then:

\begin{enumerate}
    \item 
There is no $\pi \in \C$ such that $maj(\pi)=k+1$. 
    \item 
There is no $\pi \in \C$ such that $maj(\pi)={n \choose 2} -1$. 

\end{enumerate}
\end{lemma}

\begin{proof}

According to Proposition \ref{transfer to tableaux}, it is sufficient to prove that there is no $T \in SYT_n(0)$ such that $maj(T)\in \{k+1,{n \choose 2} -1\}$. 
\begin{enumerate}
    \item 
    First, since there are no odd columns in any diagram of $D_n(0)$ and $n=2k$, the number of columns does not exceed $k$. If there is some $T\in SYT_n(0)$ such that $maj(T)=k+1$, then we have either $Des(T)=\{k+1\}$ or $Des(T)$ contains elements less than $k+1$, the sum of which is $k+1$. 
    In the first case, all the elements $1,\dots,k+1$ occupy the first line, which means that $T$ has more than $k$ columns, a contradiction. 
    In the second case, $k+1,\dots ,2k$ must be placed in different columns. Since all the columns contain at least $2$ elements, the numbers $1,\dots, k$ are also placed in different columns but this means that the major index is $k$ and not $k+1$.

    \item We have already seen in Lemma \ref{min and max on $D_n(r)$} that the maximal value, $M={n \choose 2}$ is attained by the one-column diagram. Any other diagram in $D_n(0)$ contains at least two columns, all of them even. If $T\in SYT_n(0)$ is such that $maj(T)={n \choose 2}-1=2+\cdots+ (n-1)$, then we must have $Des(T)=\{2,\dots,n-1\}$. The element $1$ must be placed in the upper left corner and since $1\notin Des(T)$, the element $2$ must be placed in the first line, right after the element $1$. Since $2 \in Des(T)$, we must put $3$ below $1$ and this argument proceeds further and gives us the tableau of Fig. \ref{max minus 1} which in not in $D_n(0)$.

\begin{figure}[!ht]\label{max minus 1}
\begin{center}

$\lambda=$ \begin{ytableau}
 {1} & {2}   \\
  {3} \\
  {4} \\
  {\cdot} \\
  {\cdot} \\
  {\cdot } \\
  {n}\\   
\end{ytableau}
\caption{}  
\end{center}
\end{figure}

 \end{enumerate}

\end{proof}

\begin{definition}
Recall that a diagram of the form $(u,1^{n-u})$ is called a hook. If the length of the leg of $\lambda$, $n-u+1$ is odd then $\lambda$ will be called an {\it odd hook}. It will be called an {\it even hook} otherwise.  
\end{definition}
The following lemma is a direct consequence of Proposition \ref{transfer to tableaux} and Lemma \ref{min and max on $D_n(r)$}.2.
\begin{lemma}\label{odd hook attains maximum}
Let $\mu$ be the conjugacy class of involutions with exactly $r$ fixed points. Then the maximal value of major index over $\mu$ is the maximal value of $D_n(r)$ which is obtained by the hook $\lambda=(r,1^{n-r})$.

\end{lemma}
The following lemma will be of great use when we prove our main result in the next section. 

\begin{lemma}\label{gap inside one tableau}

Let $n\geq 6$ and let $\lambda \vdash n$ be such that $\lambda\neq (1^n)$ and $\lambda \neq (n)$. Then  $M(\lambda)-m(\lambda) \geq 4$. 

Moreover, if $\lambda=(a^b)$ is a rectangle then $M(\lambda)-m(\lambda) \geq 6$.
\end{lemma}
\begin{proof}
 
Recall from lemma \ref{Max and Min} that $m(\lambda)=b(\lambda)$, while $M(\lambda)={n \choose 2}-b(\lambda') $, so that $\Delta:=M(\lambda)-m(\lambda)={n \choose 2}-(b(\lambda)+b(\lambda'))$. 

Let $T_1$ be the tableau of shape $\lambda$, filled by the numbers $0,\dots, n-1$ in such a way that they can be read in this order along the consecutive rows, increasing along each row.  

Let $T_2$ be the tableau of shape $\lambda$, filled in such a way that the square placed in row $i$ and column $j$ contains $i+j$ where the rows and columns are numbered $0,1,\dots$, etc.

Here is an example with $\lambda=(4,3,1,1)$:

$T_1=$
\ytableausetup {boxsize=3em}
\begin{ytableau}
 {0} & {1} & {2} & {3}  \\
  {4} & {5} & {6}\\
  {7} \\
  {8} \\
\end{ytableau} 
\ytableausetup {boxsize=3em}
$T_2=$\begin{ytableau}
 {0+0} & {0+1} & {0+2} & {0+3}  \\
  {1+0} & {1+1} & {1+2}\\
  {2+0} \\
  {3+0} \\
\end{ytableau} 
\ytableausetup {boxsize=normal}

Note that the sum of elements of $T_1$ is ${n \choose 2}$ while the sum of elements of $T_2$ is $b(\lambda)+b(\lambda')$, so that $\Delta$ is the difference of the row sums.  
Denote for each $i$ by $\Delta_i$
the difference between the sum of elements of row number $i$ in $T_1$ and the sum of elements of row number $i$ in $T_2$, so that $\Delta=\sum\limits_{i}\Delta_i$. 

It is easy to see that $\Delta_0=0$ and $\Delta_1=\lambda_1(\lambda_0-1)$.  Moreover, for each $i \geq 2$, we have $\Delta_i \geq 0$
since the first element of each such row in $T_1$ is greater than the corresponding element in $T_2$ and in both tableaux all elements of row $i$ are consecutive.

By the assumptions of the lemma, we must have $\lambda_1 \geq 1$. we divide in several cases according to the value of $\lambda_1$:
\begin{itemize}
    \item If $\lambda_1 \geq 3$ then $\lambda_0 -1 \geq 2$ so that $\Delta \geq \Delta_1 \geq 6$ and we are done. 
\item  If $\lambda_1=2$ then we must have $\lambda_0 \geq 2$. Now, if $\lambda_0=2$ then $\lambda$ contains the diagram $(2,2,1)$ which satisfies $\Delta \geq 4$ by inspection. 

The case $\lambda_1=2$ and $\lambda_0=3$ is symmetric to the last case, i.e $\Delta$ is invariant under transposition. Furthermore, if $\lambda_0 \geq 4$ then $\Delta_1 \geq 6$ and we are done.

\item if $\lambda_1=1$ then we must have $\lambda_0 \geq 2$ since otherwise we get $\lambda=(1^n)$ which contradicts the assumptions.  If $\lambda_0=2$ then $\lambda$ contains the diagram $(2,1^4)$, since by the assumption $n \geq 6$, so that $\Delta \geq 4$ by inspection.

\end{itemize}

Regarding the second claim, if $\lambda=(a^b)$ is a rectangle, then since $n \geq 6$, it must contain either the shape $(3^2)$ or the shape $(2^3)$. If $\lambda$ contains the shape $(3^2)$ then it is easy to see that $\Delta \geq \lambda_1(\lambda_0-1) \geq 6$. The other case follows by symmetry.
\end{proof}

\section{The range of the major index on conjugacy classes of involutions}\label{main results}
Our main result is the following:
\begin{thm}\label{main theorem itself}
Let $\mu=(2^k,1^r)$ be a partition of $n$ and let $\cmu$ be the corresponding conjugacy class of involutions in $\sn$. Then 
\begin{itemize}
    \item If $r\neq 0$, then the major index on $\cmu$ attains all values between $k$ and $\binom{n}{2}-\binom{r}{2}$. 
    
    \item If $r=0$, then the major index on $\cmu$ attains all the values mentioned in the first clause, excluding $k+1$ and $\binom{n}{2}-1$.
    \end{itemize}
  Moreover, any other value outside this range is not attained. 
\end{thm}
We present here the sketch of the proof and postpone the details to Section \ref{detailed proofs}. 
\subsection{Sketch of the proof}
By Propositions \ref{shuz} and \ref{transfer to tableaux}, it is sufficient to prove our results for the sets $SYT_r(n)$, consisting of all standard tableaux of shapes having exactly $r$ odd columns. 

This will be done by ordering the set $D_n(r)$ of diagrams of size $n$ with exactly $r$ odd columns, (see Definition \ref{r odd}) by reverse dominant order (see Definition 2.2.2 in \cite{Sagan}) and presenting an algorithm which starts with the diagram $\lambda^0=(n-k,k)=(k+r,k)$, attaining the minimum value of $maj$ over $D_n(r)$ which is $k$ and ends with the odd hook diagram $\lambda^e=(r,1^{2k})$, attaining the maximum value of $maj$ over $D_n(r)$ which is ${n \choose 2} -{r \choose 2}$ (see Lemma \ref{min and max on $D_n(r)$}).  The algorithm traverses the set $D_n(r)$ in such a way that in each step one or two squares of a diagram $\lambda \in D_n(r)$ are transferred to a new place to obtain another diagram $\nu\in D_n(r)$ such that following condition is satisfied:

\begin{equation}\label{inequality}
M(\lambda)\geq m(\nu),    
\end{equation}
where $M(\lambda)$ ($m(\nu)$) is the maximum (minimum) value of $maj$ on $SYT(\lambda)$ ($SYT(\nu)$), respectively as in Proposition \ref{Max and Min}. 
Together with Proposition \ref{Max and Min}, we are done, apart from some special cases that will be treated separately in Section \ref{detailed proofs}. 
\subsection{The algorithm}

%\begin{center}
%{\bf The algorithm:}
%\end{center}
Let $\lambda=(\lambda_0,\lambda_1,\dots,\lambda_t)$. We add infinite number of zeroes at the end of $\lambda$ and write $\lambda$ as a Young diagram.
Also, for each $i$, denote the last square of the row $\lambda_i$ by $\lambda_i^*$ . 
Now perform the following steps:

\begin{enumerate}

    \item If $\lambda$ is an odd hook, then we are done by Lemma \ref{min and max on $D_n(r)$}.2. 
    
    \item If $\lambda$ is an even hook, then we distinguish between two cases:
    \begin{itemize}
    \item If $\lambda=(1^{2k})$, i.e. $r=0$, then again we are done by Lemma \ref{min and max on $D_n(r)$}.2.

     \item Otherwise, let $\nu$ be the shape obtained from $\lambda$ by removing the square $\lambda_0^*$ and placing it at the end of the first column of $\lambda$. This $\nu$ is an odd hook so we are back in (1)  (see the passage from $\lambda^7$ to $\lambda^8$ in Example \ref{main example}). We will justify this step in  Lemma \ref{correctness in step 2}
     \end{itemize} 
    \item If there is some $0 \leq j \leq t$ such that $\lambda_{j}>\lambda_{j+1}>\lambda_{j+2}$, then let $i$ be maximal with respect to this property. Since $\lambda_i>\lambda_{i+1}>\lambda_{i+2}$, we have in $\lambda$ a column of length $i+1$ and a column of length $i+2$ (see Fig. \ref{first case}).   
    Now, remove $\lambda_i^*$ and place it as the last square of row $i+2$ and let $\nu$ be the resulting shape.
    We will elaborate on the utility of this step in Lemma \ref{correctness in step 3}

This case is illustrated in Figs. \ref{first case} and \ref{first case 2 }.

\begin{figure}[!ht]
\begin{center}

$\lambda=  $\begin{ytableau}
 {} & {} & {} & {*}  \\
  {} & {} & {}\\
  {} \\
  {} \\
\end{ytableau}
$\longrightarrow \nu= $ \begin{ytableau}
 {} & {} & {}   \\
  {} & {} & {}\\
  {} & {*}\\
  {} \\
\end{ytableau}

\caption{\label{first case} $\lambda=(4,3,1,1)$ and $\nu=(3,3,2,1), i=0$.}
\end{center}
\end{figure}

\begin{figure}[!ht]
\begin{center}

$\lambda=  $\begin{ytableau}
 {} & {} & {} \\
  {} & {*} \\
  {} \\
\end{ytableau}
$\longrightarrow \nu= $ \begin{ytableau}
 {} & {} & {}   \\
  {} \\
  {} \\
  {*} \\
\end{ytableau}

\caption{\label{first case 2 } $\lambda=(3,2,1)$ and $\nu=(3,1,1,1), i=1$.}
\end{center}
\end{figure}
    
\item Otherwise, if there is no $0 \leq j \leq t$ such that $\lambda_j>\lambda_{j+1}>\lambda_{j+2}$, then there must exist some $j$ such that  $\lambda_j>\lambda_{j+1}=\lambda_{j+2}$ (the existence of such $j$ is guaranteed since we can always choose $j=t$ to get $\lambda_t>\lambda_{t+1}=0=\lambda_{t+2}$). Let $i$ be maximal with respect to this property and such that $\lambda_i>1$  (recall that the case $\lambda_i=\lambda_j=1$ for each $i,j$ was treated earlier in cases (1) and (2)). 
 
  For the $i$ we chose, we must have $\lambda_{i-1}=\lambda_i$ since otherwise, we have already treated this case in (3). This means that the squares $\lambda_{i-1}^*$ and $\lambda_i^*$ form a vertical domino  (see Figures \ref{second case} . We distinguish between two sub-cases. 
\begin{enumerate}
   
   \item If we have $\lambda_i-\lambda_{i+1}=1$ then, by the maximality of $i$ we must have $\lambda_i=2$. In this case, we put the domino at the end of the first column (see Figure \ref{second case}, Left). 
    \item Now, if $\lambda_i-\lambda_{i+1}>1$, then we move that domino to the pair of squares located right after $\lambda_{i+1}^*$ and $\lambda_{i+2}^*$  (see Figure \ref{second case}, Right).
    
\end{enumerate}

We justify these steps in Lemma \ref{correctness in step 4}. 

\item Back to step $(1)$ with $\lambda=\nu$. 
\end{enumerate}

\begin{figure}[!ht]
\begin{center}

\begin{ytableau}
 {} & {}  & {} \\
  {} & {}  \\
  {} & {*} \\
  {} & {*}  \\
  {} \\
  {}   \\
  {}
\end{ytableau}
$\longrightarrow$
\begin{ytableau}
 {} & {} & {}  \\
  {} & {} \\
  {} \\
  {} \\
  {}   \\
  {}\\
  {} \\
  {*} \\
  {*}
\end{ytableau}
\hspace{1cm}
\begin{ytableau}
 {} & {}  & {}\\
  {} & {} & {*} \\
  {} & {} & {*} \\
  {}  \\
  {} \\
  {}  
\end{ytableau}
$\longrightarrow$
\begin{ytableau}
 {} & {} & {}  \\
  {} & {} \\
  {} & {} \\
  {} & {*} \\
  {} & {*}   \\
  {}\\
\end{ytableau}

%\caption{\label{second case}}
\caption{Left: $i=3$, Right: $i=2$}\label{second case}
\end{center}
\end{figure}

%\begin{figure}[!ht]
%\begin{center}

%\caption{\label{second case}}
%\caption{$i=2$}\label{third case}
%\end{center}
%\end{figure}

\begin{figure}[!ht]
\begin{center}

$\lambda^0=$ \begin{ytableau}
 {} & {} & {}  & {} & {*}  \\
  {} & {} & {} & {} 
\end{ytableau} $\rightarrow$
$\lambda^1=$ \begin{ytableau}
 {} & {} & {}  & {} \\
  {} & {} & {} & {*} \\
  {}
  \end{ytableau}$\rightarrow$
  $\lambda^2=$ \begin{ytableau}
 {} & {} & {}  & {*}  \\
  {} & {} & {} \\
  {}\\
  {}
\end{ytableau}
$\rightarrow$
$\lambda^3=$ \begin{ytableau}
 {} & {} & {}  \\
  {} & {} & {} \\
  {} & {*} \\
  {}
\end{ytableau}
$\rightarrow$
$\lambda^4=$ \begin{ytableau}
 {} & {} & {*}  \\
  {} & {} & {*} \\
  {} \\
  {}\\
  {}
\end{ytableau}
$\rightarrow$
$\lambda^5=$ \begin{ytableau}
 {} & {}   \\
  {} & {}  \\
  {} & {} \\
  {} & {*} \\
  {}
\end{ytableau}
$\rightarrow$
$\lambda^6=$ \begin{ytableau}
 {} & {}   \\
  {} & {*}  \\
  {} & {*} \\
  {} \\
  {}   \\
  {}
\end{ytableau}
$\rightarrow$
$\lambda^7=$ \begin{ytableau}
 {} & {*}   \\
  {}  \\
  {} \\
  {} \\
  {}   \\
  {}\\
  {} \\
  {}
\end{ytableau}
$\rightarrow$
$\lambda^8=$ \begin{ytableau}
 {} \\
  {}  \\
  {} \\
  {} \\
  {}   \\
  {}\\
  {} \\
  {}\\
  {}
\end{ytableau}
\caption{\label{full example}}
\end{center}
\end{figure}

\begin{example}\label{main example}
Let $n=9$ and $\mu=(2^4,1)$. Then the RSK algorithm takes the elements of $\cmu$ to $D_9(1)$ and by Lemma \ref{min and max on $D_n(r)$}.1, the diagram which attains the minimum value of $maj$ is $\lambda^0=(5,4)$. 
The process of the algorithm ends in diagram $\lambda^8$ which is an odd hook that attains the maximum value of $maj$. The process is depicted in Fig.\ref{full example}.
The data of the tableaux taking part in the process are listed in the table after that figure. 

\begin{center}
\begin{tabular}{||c | c | c  | c|c |c||} 
 \hline
 $i$ & $min(\lambda^i)$ & $max(\lambda^i)$ & Step used \\ [0.5ex] 
 \hline\hline
 0 & 4 & 20 & 3 \\ 
 \hline
 1 & 6 & 24 & 3  \\
 \hline
 2 & 8 & 27 & 3  \\
 \hline
 3 & 10 & 29 & 3 \\
 \hline
 4 & 12 & 30 & 4(b)  \\ 
 \hline
  5 & 16 & 32 & 3 \\
 \hline
  6 & 18 & 33 & 4(a) \\
 \hline
  7 & 28 & 35 & 2 \\
 \hline
  8 & 36 & 36 & 1 \\
 \hline
\end{tabular}
\end{center}
\end{example}

\section{Correctness of the algorithm} \label{detailed proofs}
In this section, we provide comprehensive proof of the correctness of the algorithm we presented in the previous section. We divide our treatment into two parts, according to the number of fixed points $r$ in the conjugacy class of involutions $\cmu$ corresponding to the partition $\mu=(2^k,1^r)$.  
We start with the case $r \geq 2$ and defer the case $r<2$ to Subsection \ref{special cases}.

The main theme of Theorem \ref{main theorem itself} is showing that in each step of the algorithm when we pass from $\lambda$ to $\nu$, we have $M(\lambda) \geq m(\nu)$. 
This inequality will be proved by one of the following methods.  
If $\lambda=(n-k,1^k)$ is an even hook (this happens in step $(2)$ of the algorithm), then we prove it directly in Lemma \ref{correctness in step 2}. 

Otherwise, we show that one of the following conditions is satisfied.
\medskip
\begin{enumerate}
    \item $$0\leq m(\nu)-m(\lambda) \leq 4$$
    
    \item $$M(\nu)-M(\lambda)= 2.$$
\end{enumerate}

These conditions, together with Lemma \ref{gap inside one tableau} imply that whenever we pass from $\lambda$ to $\nu$ in the process of the algorithm, we have $M(\lambda) \geq m(\nu)$.\medskip

(Indeed, if $0\leq m(\nu)-m(\lambda) \leq 4$ then $m(\nu) \leq m(\lambda)+4$ and by Lemma \ref{gap inside one tableau} we have $M(\lambda)-m(\lambda) \geq 4$, so $m(\lambda)+4\leq M(\lambda)$, hence $m(\nu) \leq M(\lambda)$.  The other case is similar, using Lemma \ref{gap inside one tableau} for $\nu$ rather than $\lambda$). 
For more clarification, see Fig. \ref{diagrams}.

\begin{figure}[!ht]
\label{interval}
    \centering
    \includegraphics[scale=0.75]{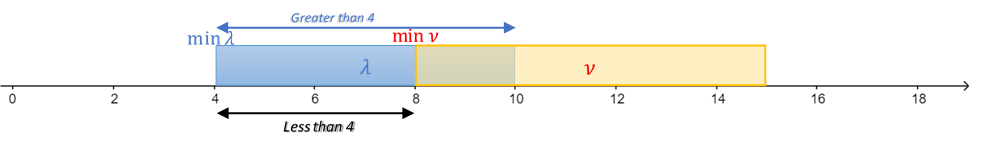}
    \caption{}
    \label{diagrams}
\end{figure}

Now, in the case where both $\lambda$ and $\nu$ are not rectangles with more than $2$ rows, the fact that $M(\lambda)\geq m(\nu)$, together with Proposition \ref{Max and Min} which assures us that $maj$ obtains every value in their ranges, guarantee that in the passage from $\lambda$ to $\nu$, all the values of $maj$ are attained. 

If, $\lambda$ is a rectangle with more than $2$ rows, then by Lemma \ref{gap inside one tableau}, we have that $M(\lambda)-m(\lambda) \geq 6$ and by Proposition \ref{Max and Min}, each value in the open integral interval $(m(\lambda)+1,M(\lambda)-1)$ is attained, so that by an argument similar to the one enclosed by the parentheses above, we immediately get that $m(\nu) \leq M(\lambda)-2$ and we proceed as before. The case where $\nu$ is a rectangle with more than $2$ rows is treated similarly. 

Note that the Algorithm does not necessarily arrive at each shape of $D_n(r)$, but it does prove that 
all the values from $m(\lambda^0)$
to $M(\lambda^e)$ are obtained by $maj$. 

We turn now to explain the validity of steps $(2),(3)$, and $(4)$ of the algorithm. 

\begin{lemma}\label{correctness in step 2}
If $\lambda=(n-2k+1,1^{2k-1})$ is an even hook and $\nu=(n-2k,1^{2k})$ is the odd hook we arrive at in step  $(2)$ of the algorithm, then $M(\lambda) \geq m(\nu)$. In other words, step $(2)$ is justified. 
\end{lemma}

\begin{proof}

We first note that the number of odd columns has not been changed by the transfer of the square $\lambda_0^*$. 

Let us write $n=2k+r$, so that $\lambda=(r+1,1^{n-r-1})$ and $\nu=(r,1^{n-r})$. We have then $M(\lambda)=\binom{n}{2}-\binom{r+1}{2}$ and $m(\nu) =\binom{n-r+1}{2}$ and we have to show that $\binom{n}{2}-\binom{r+1}{2} \geq \binom{n-r+1}{2}$. 

 Showing this inequality is equivalent to showing that $(n-r)^2+r^2 \leq n^2-2n$ which is equivalent in turn to $(n-r-1)(r-1) \geq 1$ .
 
  The last inequality holds if and only if $r \geq 2$, which is the working assumption of this section
 ($n-r \geq 2$ is obvious since $\lambda$ is an even hook). 
 The case $r<2$ will be discussed in Section \ref{special cases}.

\end{proof}

    \begin{lemma}\label{correctness in step 3}
    If we move a square in the diagram $\lambda$ to get the diagram $\nu$ according to the conditions of step $(3)$ of the algorithm, then $M(\lambda)\geq m(\nu)$, i.e.
step (3) is justified. 
    \end{lemma}
    
    \begin{proof}
    The length $i+1$ of the column ending with the square $\lambda^*_i$ is reduced by 1, while the length of the column in $\nu$ to which we added the square, increases from $i+1$ to $i+2$, so that an odd column becomes an even column and vice versa. There is no change in any other columns.
    
     Note that for each $j\notin\{i,i+2\}$ we have $\nu_j=\lambda_j$, and also $\nu_i=\lambda_i-1$ and $\nu_{i+2}=\lambda_{i+2}+1$, so we have  
     
     $$b(\nu)=i(\lambda_i-1)+(i+2)(\lambda_{i+2}
+1)+\sum_{j\neq i,i+2}{j\lambda_j}=$$ $$i\lambda_i+(i+2)\lambda_{i+2}+2+\sum_{j \neq i,i+2}j\lambda_j=b(\lambda)+2,$$ so that  $m(\nu)-m(\lambda)=2\le 4$. Hence, by the discussion in the beginning of this section, we conclude that $M(\lambda)\ge m(\nu)$.

\end{proof}

\begin{lemma}\label{correctness in step 4}

If we move a domino in the diagram $\lambda$ to get the diagram $\nu$ according to the conditions of either one of the two cases of step $(4)$ of the algorithm, then we have $M(\lambda)\geq m(\nu)$, i.e.
step $(4)$ is justified. 
\end{lemma}

\begin{proof}

Since we move two squares from one column to another, the number of odd columns is not changed in the passage from $\lambda$ to $\nu$. 

We show now that in case (a) we have $0 \leq M(\nu) - M(\lambda) \leq 4$, while in case (b) we have $0 \leq m(\nu)-m(\lambda) \leq 4$ .

\begin{itemize}
\item In case (a) we have:
$$M(\nu)-M(\lambda)={n \choose 2}-b(\nu')-\left({n \choose 2}-b(\lambda')\right)=b(\lambda')-b(\nu').$$ 
Since $\lambda_{i-1}^*=\lambda_i^*=2$, the domino is moved from the second column to the first one and thus the columns sum of $\lambda$ is decreased by $2$, while the columns sum of $\nu$ gains nothing, so that $b(\lambda')-b(\nu')=2$ (see Remark \ref{graphic blambda}). 

\item In case (b), we have:

$$b(\nu)=(i-1)(\lambda_{i-1}-1)+i(\lambda_i-1)+(i+1)(\lambda_{i+1}+1)+(i+2)(\lambda_{i+2}+1)+\sum_{j<i-1}{j\cdot \lambda_j}+\sum_{i+2<j}{j\cdot \lambda_j}=$$

$$(i-1)\lambda_{i-1}+i\lambda_i+(i+1)\lambda_{i+1}+(i+2)\lambda_{i+2}+4+\sum_{j<i-1}{j\cdot \lambda_j}+\sum_{i+2<j}{j\cdot \lambda_j}=b(\lambda)+4$$

so that $m(\nu)-m(\lambda)=b(\nu)-b(\lambda)=4$.
\end{itemize}
\end{proof}
\subsection{Special cases}\label{special cases}
We deal now with the case $r \leq 1$. A problem might emerge if we arrive at the diagram $\nu=(1^{n})$ for $n$ odd, in which case, the inequality $M(\lambda)\geq m(\nu)$ does not hold (see for example the passage from $\lambda^7$ to $\lambda^8$ in Example \ref{main example}). There are two possibilities of diagrams $\lambda$ from which we might arrive at such a situation.  

\begin{enumerate}
    \item[(a)] $\lambda=(2,1^{n-2})$.  This case might appear in step $(2)$ of the algorithm, where Lemma \ref{correctness in step 2} will not hold true since $M(\lambda)={n \choose 2} -b(\lambda')={n \choose 2}-1$, while $m(\nu)={n \choose 2}$. 
     Nevertheless, this does not harm the continuity of the major index.
    
    \item[(b)] $\lambda=(2^2,1^{n-4})$. This can happen in the first case of step $(4)$ of the algorithm.  
    
    In this case, we have $M(\lambda)={n \choose 2}-2$, while $m(\nu)={n \choose 2}$ (although we do have $M(\nu)-M(\lambda)=2$, the implication that $M(\lambda) \geq m(\nu)$ which is mentioned in the parenthesized paragraph before Fig. \ref{diagrams} does not hold, since the use of Lemma \ref{gap inside one tableau} with respect to $\nu$ is not allowed there). 
    
    To complete this gap, we deviate slightly from the algorithm: instead of removing the domino in the right, we remove only the bottom square of the domino, to get $\nu'=(2,1^{n-2})$ in which $M(\lambda)-M(\nu')={n \choose 2}-\left({n \choose 2}-1\right)=1$  (note that indeed $\nu' \in D_n(1)$ since $n=2k+1$ is odd). The passage from $\nu'$ to $\nu=(1^n)$ is identical to the one described in the former case (a) (see figure \ref{last example})
\end{enumerate}

\begin{figure}[!ht]
\begin{center}
$\lambda=$ \begin{ytableau}
 {}  & {}  \\
  {}  & {*} \\
  {} \\
  {}\\
  {}
\end{ytableau}
$\rightarrow$
$\nu'=$ \begin{ytableau}
 {} & {*}   \\
  {}  \\
  {}  \\
  {} \\
  {}\\
  {}
\end{ytableau}
$\rightarrow$
$\nu=$ \begin{ytableau}
 {}   \\
  {}  \\
  {} \\
  {} \\
  {}   \\
  {}  \\
  {}
\end{ytableau}
\caption{\label{last example}}
\end{center}
\end{figure}

We get now to the case $r=0$, here we have two types of problems.
\begin{enumerate}
    \item 
    If $\lambda=(2^2,1^{n-4})$ and $\nu=(1^n)$, then we have $M(\lambda)={n \choose 2}-2$, 
    while $m(\nu)={n \choose 2}$. The value ${n \choose 2}-1$ does not appear, in accordance with the claim of Theorem \ref{main theorem itself}. Actually, the value ${n \choose 2}-1$ is not obtained at all due to Lemma \ref{missing}.  
    
    \item If $\lambda^0=(k,k)$ which by Lemma \ref{min and max on $D_n(r)$}.1 obtains the minimum value of major index, then since $\lambda^0$ is a rectangle, by Proposition \ref{Max and Min}, $m(\lambda^0)+1$ is not obtained. Again, this is in accordance with the claim of Theorem \ref{main theorem itself}. Note that this value will not be obtained elsewhere by Lemma \ref{missing}.    
    
  \end{enumerate}

\section*{Acknowledgements}
The results of this paper are part of the Ph.D thesis of the second author, supervised by Ron Adin and Yuval Roichman.
We thank both of them for fruitful conversations.


\begin{thebibliography}{}

\bibitem{Sw} S. Billey, M. Konvalinka and J. P. Swanson,
{Tableau posets and the fake degrees of coinvariant algebras,
Advances in Mathematics,
Volume 371,
2020,
107252}

\bibitem{Green} J. A. Green. {\it The characters of the finite general linear groups}, Trans. Amer.
Math. Soc., 80:402–447, 1955.


\bibitem{Mc}
MacMahon, P. A. . {\it The indices of permutations and the derivation therefrom of functions of a single variable associated with the permutations of any assemblage of objects},  Amer. J. Math. 35 (3),(1913), 281–322.



\bibitem{Sagan} 
B.\ E.\ Sagan, 
{\em The symmetric group: Representations, combinatorial algorithms, and symmetric functions}, 
Second edition, Graduate Texts in Math., no.~203, Springer-Verlag, New York, 2001.


%\bibitem{OEIS} 
%N.\ J.\ A.\ Sloane, 
%The On-line Encyclopedia of Integer Sequences.

\bibitem{St2} 
R.\ P.\ Stanley, 
Enumerative Combinatorics, Vol. 2, 
Cambridge Studies in Adv.\ Math., no.~62. Cambridge Univ.\ Press, Cambridge, 1999.
\bibitem{Sta79}
R.\ P.\  Stanley  {\it Invariants of finite groups and their applications to combinatorics},
Bull. Amer. Math. Soc. (N.S.), 1(3):475–511, 1979.


\bibitem{S}M.\ P.\ Sch\"{u}tzenberger, {\em La correspondence de Robinson}, in: D. Foata, Ed., Combinatoire et
Representation du Groupe Sym6trique, Lecture Notes in Mathematics No. 579 (Springer,
Berlin, 1977) 59-113. 

\bibitem{Ste} R. Steinberg, {\it A geometric approach to the representations of the full linear group
over a Galois field}, Trans. Amer. Math. Soc., 71:274–282, 1951.



\end{thebibliography}
\end{document}